\documentclass[aop,MSNbibl,seceqn,dvips]{arximspdf}
\usepackage{graphicx}
\doi{10.1214/12-AOP791} 
\volume{42}
\issue{1}
\pubyear{2014}
\firstpage{146}
\lastpage{167}

\makeatletter
\newcommand\Var{\operatorname{Var}}
\newcommand\R{{\mathbf{R}}}
\newcommand\N{{\mathbf{N}}}
\renewcommand\P{{\mathbf{P}}}
\newcommand\E{{\mathbf{E}}}
\newcommand\tr{\operatorname{tr}}
\newcommand\I{{\mathbf{I}}}

\newcommand\eps{\varepsilon}

\newcommand\Ba{\mathbf{a}}
\newcommand\Bb{\mathbf{b}}
\newcommand\Be{\mathbf{e}}
\newcommand\Bv{\mathbf{v}}
\newcommand\Bw{\mathbf{w}}
\newcommand\Bx{\mathbf{x}}

\newcommand\BI{\mathbf{I}}
\newcommand\BN{\mathbf{N}}

\newcommand\CE{\mathcal{E}}
\newproclaim{notation}{Notation}

\newcommand{\eqref}[1]{(\ref{#1})}
\newtheorem{theorem}[subsection]{Theorem}
\newtheorem{lemma}[subsection]{Lemma}
\newproclaim{remark}[subsection]{Remark}

\makeatother

\begin{document}
\begin{frontmatter}

\title{Random matrices: Law of the determinant}
\runtitle{Law of the determinant}

\begin{aug}
\author[A]{\fnms{Hoi H.} \snm{Nguyen}\corref{}\ead[label=e1]{nguyen.1261@math.osu.edu}}
\and
\author[B]{\fnms{Van} \snm{Vu}\ead[label=e2]{van.vu@yale.edu}}
\runauthor{H. Nguyen and V. Vu}
\affiliation{University of Pennsylvania and Yale University}
\address[A]{Department of Mathematics\\
Ohio State University\\
231 West 18th Avenue\\
Columbus, Ohio 43210\\
USA\\
\printead{e1}}
\address[B]{Department of Mathematics\\
Yale University\\
New Haven, Connecticut 06520\\
USA\\
\printead{e2}}
\end{aug}

\received{\smonth{11} \syear{2011}}
\revised{\smonth{6} \syear{2012}}

%
\begin{abstract}
Let $A_n$ be an $n$ by $n$ random matrix whose entries are independent
real random variables with mean zero, variance one and with
subexponential tail. We show that
the logarithm of $|\det A_n|$ satisfies a central limit theorem. More precisely,
\begin{eqnarray*}
&&\sup_{x\in\R} \biggl|\P \biggl(\frac{\log(|\det A_n|)- ({1}/{2}) \log
(n-1)!}{\sqrt{({1}/{2}) \log n}}\le x \biggr) -\P\bigl(\BN(0,1)
\le x \bigr) \biggr| \\
&&\qquad\le\log^{-{1}/{3} +o(1)} n.
\end{eqnarray*}
\end{abstract}

%
\begin{keyword}[class=AMS]
\kwd{60B20}
\kwd{60F05}
\end{keyword}
\begin{keyword}
\kwd{Random matrices}
\kwd{random determinant}
\end{keyword}

\end{frontmatter}

\section{Introduction}\label{sectionintroduction}

Let $A_n$ be an $n$ by $n$ random matrix whose entries $a_{ij}, 1\le
i,j\le n$, are independent real random variables of zero mean and unit variance.
We will refer to the entries $a_{ij}$ as the \textit{atom} variables.

As determinant is one of the most fundamental matrix functions,
it is a basic problem in the theory of random matrices to study the distribution
of $\det A_n$ and indeed this study has a long and rich history.
The earliest paper we find on the subject is a paper of Szekeres and
Tur\'an \cite{SzT} from 1937, in which they studied an extremal
problem. In the 1950s, there is
a series of papers \cite{FT,NRR,Turan,Pre} devoted to the computation
of moments of fixed orders of $\det A_n$ (see also \cite{Gbook}). The
explicit formula for higher moments gets very complicated and is in
general not available, except in the case when the atom variables have
some special distribution (see, e.g., \cite{Dembo}).

One can use the estimate for the moments and Markov inequality to
obtain an upper bound on $|\det A_n |$. However, no lower bound was
known for a long time. In particular, Erd\H{o}s asked whether $\det
A_n$ is nonzero with probability tending to one. In 1967, Koml\'os
\cite{Kom,Kom1} addressed this question, proving that almost surely
$|\det A_n | > 0$ for random Bernoulli matrices (where the atom
variables are i.i.d. Bernoulli, taking values $\pm1$ with probability
$1/2$). His method also works for
much more general models. Following \cite{Kom}, the upper bound on the
probability that
$\det A_n =0$ has been improved in \cite{KKS,TVdet,TVsing,BVW}.
However, these results do not say much about the value of $|\det A_n |$ itself.

In a recent paper \cite{TVdet}, Tao and the second author proved that
for Bernoulli random matrices,
with probability tending to one (as $n$ tends to infinity),
%
\begin{equation}
\label{TVlow} \sqrt{n !} \exp( -c \sqrt{ n \log n} ) \le|\det A_n |
\le \sqrt{n!} \omega(n)
\end{equation}
for any function $\omega(n)$ tending to infinity with $n$. This shows
that almost surely, $\log|\det A_n |$ is $ (\frac{1}{2} +o(1) )n \log
n$, but does not provide any distributional information. For related
works concerning other models of random matrices, we refer to~\cite{Ro}.

In \cite{Goodman}, Goodman considered random Gaussian matrices where
the atom variables are i.i.d. standard Gaussian variables. He noticed that
in this case the determinant is a product of independent Chi-square
variables. Therefore, its logarithm is the sum of independent variables
and, thus, one expects
a central limit theorem to hold. In fact, using properties of
Chi-square distribution, it is not very hard to prove
%
\begin{equation}
\label{Girko} \frac{\log(|\det A_n|)- ({1}/{2})
\log(n-1)!}{\sqrt{({1}/{2}) \log n}} \rightarrow\BN(0,1).
\end{equation}

We refer the reader to \cite{RW}, Section 4, for further discussion on
this model.

In \cite{G2}, Girko stated that \eqref{Girko} holds
for general random matrices under the additional assumption
that the fourth moment of the atom variables is 3. Twenty years later,
he claimed a
much stronger result which replaced the above assumption by the
assumption that the atom variables have bounded $(4 +\delta)$th moment
\cite{G}. However, there are points which are not clear in these
papers and we have not found any researcher who can explain the whole
proof to us.
In our own attempt, we could not pass the proof of Theorem 2 in \cite
{G}. In particular, definition (3.7) of this paper requires the matrix
$\Xi\bigl({1 \atop k}\bigr) $ to be invertible, but this assumption can easily fail.

In this paper, we provide a transparent proof for the central limit
theorem of
the log-determinant. The next question to consider, naturally, is the
rate of convergence.
We are able to obtain a rate which we believe to be
near optimal.

We say that a random variable $\xi$ satisfies condition \textup{C0}
(with positive constants $C_1, C_2$) if
%
\begin{equation}
\label{eqntailbound} P\bigl(|\xi|\ge t\bigr) \le C_1\exp\bigl(-t^{C_2}
\bigr)
\end{equation}
for all $t>0$.

\begin{theorem}[(Main theorem)]\label{theoremmain}
Assume that all atom variables $a_{ij}$ satisfy condition \textup{C0}
with some positive constants $C_1, C_2$. Then
%
\begin{equation}
\label{lim1}\qquad\quad\!\!  \sup_{x\in\R} \biggl|\P \biggl(\frac{\log(|\det A_n|)- ({1}/{2})
\log(n-1)!}{\sqrt{({1}/{2} )\log n}}\le x \biggr) -
\Phi(x) \biggr| \le\log^{-1/3 +o(1)} n.
\end{equation}
\end{theorem}

Here and later, $\Phi(x) = \P( \N(0,1) < x) = \frac{1}{\sqrt{2\pi
}} \int_{-\infty}^x \exp(-t^2/2) \,dt$.
In the remaining part of the paper, we will actually prove the
following equivalent form:
%
\begin{equation}
\label{lim2}\quad  \sup_{x\in\R} \biggl|\P \biggl(\frac{\log(\det A_n^2)- \log
(n-1)!}{\sqrt{2\log n}}\le x \biggr) -
\Phi(x) \biggr| \le\log^{-1/3
+o(1)} n.
\end{equation}

The reader is invited to consult Figure \ref{fig1} for our simulation.
To give some feeling about \eqref{lim2}, let us consider the case when
$a_{ij}$ are i.i.d. standard Gaussian. For $0\le i \le n-1$, let $V_i$ be
the subspace generated by the first $i$ rows of $A_n$. Let $\Delta_{i+1}$ denote the
distance from $\Ba_{i+1}$ to $V_{i}$, where $\Ba_{i+1}=(a_{i+1,1},\ldots, a_{i+1,n})$ is the $(i+1)$th row vector of
$A_n$. Then, by the ``base times height'' formula, we have
%
\begin{equation}
\label{eqnbase-times-height} \det A_n^2=\prod
_{i=0}^{n-1} \Delta_{i+1}^2.
\end{equation}

\begin{figure}

\includegraphics{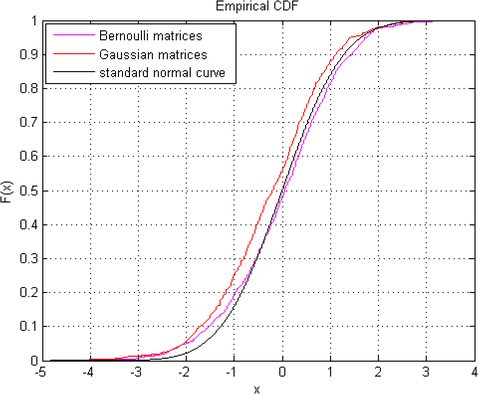}

\caption{The plot compares the distributions of $(\log(\det A^2)-\log
(n-1)!)/\sqrt{2\log n}$
for random Bernoulli matrices, random Gaussian matrices and $\N(0,1)$.
We sampled 1000 matrices of size 1000 by 1000 for each ensemble.}\label{fig1}
\end{figure}

Therefore,
%
\begin{equation}
\label{eqnbase-times-height1} \log\det A_n^2=\sum
_{i=0}^{n-1} \log\Delta_{i+1}^2.
\end{equation}

As the $a_{ij}$ are i.i.d. standard Gaussian, $\Delta_{i+1}^2$ are
independent Chi-square random variables of degree $n-i$. Thus, the
right-hand side of \eqref{eqnbase-times-height1} is a sum of
independent random variables. Notice that $\Delta^2_{i+1}$ has mean
$n-i$ and variance $O(n-i)$ and is very strongly concentrated. Thus,
with high probability $\log\Delta_{i+1}^2$ is roughly $\log((n-i) +
O(\sqrt{n-i}) )$ and so it is easy to show that $\log\Delta_{i+1}^2$
has mean close to $\log(n-i)$ and variance $O( \frac{1}{n-i})$. So
the variance of $\sum_{i=0}^{n-1} \log\Delta_{i+1}^2$ is $O(\log
n)$. To get the precise value $\sqrt{2 \log n}$, one needs to carry
out some careful (but rather routine) calculation, which we leave as an
exercise.

The reason for which we think that the rate $\log^{-1/3 +o(1)} n$
might be near optimal is that
(as the reader will see though the proofs)
$2 \log n$ is only an asymptotic value of the variance of $\log|\det
A_n|$. This approximation has an error term of order at least $\Omega
(1)$ and since $\sqrt{2 \log n +\Omega(1)}$ $- \sqrt{2 \log n}
=\Omega(\log^{-1/2} n)$, it seems that one cannot have rate of
convergence better than $\log^{-1/2 +o(1)} n $. It is a quite
interesting question whether one can obtain a polynomial rate by
replacing $\log(n-1)!$ and $2 \log n$ by other, relatively simple,
functions of $n$.

Our arguments rely on recent developments in random matrix theory and
look quite different from those in Girko's papers. In particular, we
benefit from
the arguments developed in \cite{TVdet,TVhard,TVlocal}.
We also use Talagrand's famous concentration inequality
frequently to obtain most of the large deviation results needed in this paper.

\begin{notation*}
We say that an event $E$ holds almost surely if $\P(E)$ tends to one
as $n$ tends to infinity.
For an event $A$, we use the subscript $\P_{\Bx}(A)$ to emphasize
that the probability under consideration is taken according to the
random vector $\Bx$.
For $1\le s \le n$, we denote by $\Be_s$ the unit vector $(0,\ldots
,0,1,0,\ldots,0)$, where all but the $s$th component are zero. All
standard asymptotic notation such as $O, \Omega, o,\ldots$ etc. are used
under the assumption that $n \rightarrow\infty$.
\end{notation*}

\section{Our approach and main lemmas}\label{sectionapproach}

We first make two extra assumptions about $A_n$. We assume that the
entries $a_{ij}$ are bounded in absolute value by $\log^{\beta} n$
for some constant $\beta>0$ and $A_n$ has full rank with probability
one. We will prove Theorem \ref{theoremmain} under these two extra
assumptions. In \hyperref[appendixmodel]{Appendix}, we will explain why we can implement these assumptions
without violating the generality of Theorem~\ref{theoremmain}.

\begin{theorem}[(Main theorem with extra assumptions)]\label{theoremmain1}
Assume that all atom variables $a_{ij}$ satisfy condition \textup{C0}
and are bounded in absolute value by $\log^{\beta} n$ for some
constant $\beta$. Assume furthermore that $A_n$ has full rank with
probability one. Then
%
\begin{equation}
\label{lim1} \qquad\sup_{x\in\R} \biggl|\P \biggl(\frac{\log(|\det A_n|)- ({1}/{2})
\log(n-1)!}{\sqrt{({1}/{2}) \log n}}\le x \biggr) -
\Phi(x) \biggr| \le\log^{-1/3 +o(1)} n.
\end{equation}
\end{theorem}

In the first, and main, step of the proof, we prove the claim of
Theorem \ref{theoremmain1}
but with the last $\log^{\alpha} n$ rows being replaced
by Gaussian rows (for some properly chosen constant $\alpha$). We
remark that the replacement trick was also used in \cite{G}, but for
an entirely different reason. Our reason here is that for the last few
rows, Lemma \ref{lemmaTalagrand} is not very effective.

\begin{theorem}\label{theoremmainweak} For any constant $\beta>1$
the following holds for any sufficiently large constant $\alpha>0$.
Let $A_n$ be an $n$ by $n$ matrix whose entries $a_{ij}, 1\le i\le n_0,
1\le j \le n$, are independent real random variables of zero mean, unit
variance and absolute values at most $\log^\beta n$. Assume
furthermore that $A_n$ has full rank with probability one and the
components of the last $\log^\alpha n$ rows of $A$ are independent
standard Gaussian random variables. Then
%
\begin{equation}
\label{lim3} \sup_{x\in\R}\biggl |\P \biggl(\frac{\log(\det A_n^2)- \log
(n-1)!}{\sqrt{2 \log n}}\le x \biggr) -
\Phi(x) \biggr| \le\log^{-1/3 +o(1)}n.
\end{equation}
\end{theorem}

In the second (and simpler) step of the proof, we carry out a
replacement procedure, replacing the
Gaussian rows by the original rows one at a time,j and show that the
replacement does not effect the
central limit theorem. This step is motivated by the Lindeberg replacement
method used in \cite{TVlocal}.

We present the verification of Theorem \ref{theoremmain1} using
Theorem \ref{theoremmainweak} in Section
\ref{sectiondeduction}. In the rest of this section, we focus on the
proof of Theorem \ref{theoremmainweak}.

Notice that in the setting of this theorem, the variables $\Delta_i$
are no longer independent. However, with some work, we can make the RHS
of \eqref{eqnbase-times-height1} into a sum of martingale differences
plus a negligible error, which lays ground for an application of a
central limit theorem of martingales. (In \cite{G}, Girko also used
the CLT for martingales via the base times height formula, but his
analysis looks very different from ours.) We are going to use the
following theorem, due to Machkouri and Ouchti \cite{MO}.

\begin{theorem}[(Central limit theorem for martingales, \cite{MO}, Theorem 1)]\label{theoremmartingale}
There exists an absolute constant $L$ such that the following holds.
Assume that $X_{1},\ldots,X_{m}$ are martingale differences with
respect to the nested $\sigma$-algebras $\CE_0,\CE_{1},\ldots, \CE_{m-1}$. Let $v_m^2:= \sum_{i=0}^{m-1}\E(X_{i+1}^2|\CE_i)$, and
$s_m^2:= \sum_{i=1}^{m}\E(X_{i}^2)$. Assume that $\E(|X_{i+1}^3||\CE_i) \le\gamma_i \E(X_{i+1}^2|\CE_i)$ with probability one for all
$i$, where $(\gamma_i)_1^m$ is a sequence of positive real numbers.
Then we have
\begin{eqnarray*}
&&\sup_{x\in\R} \biggl|\P\biggl(\frac{\sum_{0\le i<m}X_{i+1}}{s_m} <x\biggr) - \Phi(x) \biggr| \\
&&\qquad\le L
\biggl(\frac{\max\{\gamma_0,\ldots,\gamma_{m-1}\}
\log m}{\min\{s_m,2^m\}} + \E^{1/3}\biggl(\biggl|\frac{v_m^2}{s_m^2}-1\biggr|\biggr)
\biggr).
\end{eqnarray*}
\end{theorem}

To make use of this theorem, we need some preparation. Conditioning on
the first $i$ rows
$\Ba_1,\ldots, \Ba_{i}$, we can view $\Delta_{i+1}$ as the distance
from a random vector to $V_i:= \operatorname{ Span }
(\Ba_1, \ldots, \Ba_i)$. Since $A_n$ has full rank with probability
one, $\dim V_i =i$ with probability one for all $i$.
The following is a direct corollary of~\cite{TVlocal}, Lemma~43.

\begin{lemma}\label{lemmaTalagrand}
For any constant $\beta>0$ there
is a constant $C_3 >0$ depending on~$\beta$ such that the following holds.
Assume that $V\subset\R^n$ is a subspace of dimension $\dim(V) \le
n- 4$. Let $\Ba$ be a random vector whose components are independent
variables of zero mean and unit variance
and absolute values at most $\log^{\beta} n$. Denote by $\Delta$ the
distance from $\Ba$ to $V$. Then we have
\[
\E\bigl(\Delta^2\bigr)=n-\dim(V) = n-i
\]
and for any $t>0$
\[
\P\bigl(\bigl|\Delta-\sqrt{n-\dim(V)}\bigr|\ge t\bigr) =O \biggl(\exp\biggl(-
\frac{t^2}{\log^{C_3} n}\biggr) \biggr).
\]
\end{lemma}

Set
\[
n_0 := n -\log^{\alpha} n,
\]
where $\alpha$ is a sufficiently large constant (which may depend on
$\beta$). We will use shorthand $k_i$ to denote $n-i$, the
co-dimension of
$V_i$ (and the expected value of~$\Delta_i^2$),
\[
k_i:=n-i.
\]

We next consider each term of the right-hand side of \eqref
{eqnbase-times-height1} where $0\le i<n_0$. Using the Taylor
expansion, we write
\begin{eqnarray*}
\log\frac{\Delta_{i+1}^2}{k_i} &= &\log\biggl(1+\frac{\Delta_{i+1}^2-k_i}{k_i}\biggr)
\\
&= &\frac{\Delta_{i+1}^2-k_i}{k_i}-\frac{1}{2}\biggl(\frac{\Delta_{i+1}^2-k_i}{k_i}
\biggr)^2+R_{i+1}
\\
&:=& X_{i+1}- \frac{X^2_{i+1}}{2}+R_{i+1},
\end{eqnarray*}
where
\[
X_{i+1} := \frac{\Delta_{i+1}^2-k_i}{k_i}\quad \mbox{and }\quad R_{i+1} :=
\log(1+X_{i+1}) - \biggl(X_{i+1} -\frac{X_{i+1}^2}{2} \biggr).
\]

By applying Lemma \ref{lemmaTalagrand} with $t=k_i^{1/8}\ge\log^{\alpha/8} n$
and by choosing $\alpha$ sufficiently large, we have
with probability at least $1-O(\exp(-\log^2 n))$ [the probability
here is with respect to the random $(i+1)$th row, fixing the first $i$
rows arbitrarily]
%
\begin{equation}
\label{boundonX} |X_{i+1}|=O\bigl(k_i^{-3/8}
\bigr)=O\bigl((n-i)^{-3/8}\bigr) =o(1).
\end{equation}

Thus, with probability at least $1-O(\exp(-\log^2 n))$
\[
|R_{i+1}|=O\bigl(|X_{i+1}|^3\bigr)=O
\bigl((n-i)^{-9/8}\bigr).
\]

Hence, by a uniform bound, the following holds with probability at
least $1-n\cdot O(\exp(-\log^2n))=1-O(\exp(-\log^2n/2))$:
\[
\sum_{i<n_0}R_{i+1}=O\biggl(\sum
_{i< n_0}(n-i)^{-9/8}\biggr)=o\bigl(\log^{-2} n
\bigr),
\]
again by having
$\alpha$ sufficiently large.

We conclude the following:

\begin{lemma} \label{R} With probability at least $1-O(\exp(-\log^2n/2))$
\[
\frac{\sum_{i< n_0} R_{i+1}}{\sqrt{2\log n} }= o\biggl(\frac{\log^{-2}
n}{\sqrt{2\log n}}\biggr).
\]
\end{lemma}


We will need three other lemmas.

\begin{lemma}[(Main contribution)] \label{lemmamain1}
\[
\sup_{x\in\R}\biggl |\P \biggl(\frac{\sum_{i< n_0} X_{i+1}}{\sqrt {2\log n} }\le x \biggr) -\Phi(x)\biggr | \le
\log^{-1/3+o(1)} n.
\]
\end{lemma}

\begin{lemma}[(Quadratic terms)] \label{lemmamain2}
\[
\P \biggl( \biggl| \frac{ -\sum_{i< n_0} X_{i+1}^2/2 + \log n}{ \sqrt {2\log n}}\biggr | \ge\log^{-1/3+o(1)} n \biggr) \le
\log^{-1/3+o(1)}n .
\]
\end{lemma}

\begin{lemma}[(Last few rows)] \label{lemmamain3}
For any constant $0 < c< 1/100$
\[
\P \biggl( \biggl| \frac{\sum_{n_0\le i} \log({\Delta_{i+1}^2}/{(n-i)})} {\sqrt{2\log n} }\biggr | \ge\log^{-1/2 +c} n \biggr)=o \bigl(\exp
\bigl(- \log^{c/2} n\bigr)\bigr).
\]
\end{lemma}


Theorem \ref{theoremmainweak} follows from the above four lemmas and
the following trivial fact
(used repeatedly and with proper scaling):
\[
\P(A+B \le\sigma x) \le\P\bigl(A \le\sigma(x-\eps)\bigr) + \P(B \le\sigma\eps).
\]

The reader is invited to fill in the simple details using the following
observation:
\begin{eqnarray*}
&&\log\bigl(\det A_n^2\bigr) - \log(n-1)!\\
 && \qquad=\sum
_{i=0}^{n-1} \log\Delta_{i+1}^2 -
\log(n-1)!
\\
&&\qquad= \sum_{i=0}^{n-1}\log\frac{\Delta_{i+1}^2}{k_i} +
\log n! -\log (n-1)!
\\
&&\qquad= \sum_{i<n_0}\biggl(X_{i+1}-
\frac{X_{i+1}^2}{2}+R_{i+1}\biggr) + \sum_{n_0 \le
i}
\log\frac{\Delta_{i+1}^2}{k_i} + \log n
\\
&&\qquad= \sum_{i<n_0} X_{i+1} -\biggl(\sum
_{i<n_0} \frac{X_{i+1}^2}{2}-\log n\biggr) + \sum
_{i<n_0} R_{i+1} + \sum_{n_0 \le i}
\log\frac{\Delta_{i+1}^2}{k_i}.
\end{eqnarray*}

We will prove Lemma \ref{lemmamain1} using Theorem \ref
{theoremmartingale}. Lemma \ref{lemmamain2} will be verified by the
moment method and Lemma \ref{lemmamain3} by
elementary properties of Chi-square variables. The key to the proof of
Lemmas \ref{lemmamain1} and
\ref{lemmamain2} is an estimate on the entries of the projection
matrix onto
the space $V_i^{\bot}$, presented in Section \ref{sectionstep2}.

\section{\texorpdfstring{Proof of Lemmas \protect\ref{lemmamain1} and~\protect\ref{lemmamain2}: Opening}
{Proof of Lemmas 2.6 and 2.7: Opening}}\label{sectionstep1}

We recall from the previous section that $X_{i+1} = \frac{\Delta_{i+1}^2-k_i}{k_i}$. Denote by $P_i=(p_{st}(i))_{s,t}$ the projection
matrix onto
the orthogonal complement $V_i^\bot$. A standard fact in linear algebra
is
%
\begin{equation}
\label{eqnP} \tr(P_i)= \sum_{s}p_{ss}(i)=k_i\quad
\mbox{and}\quad\sum_{s,t}p_{st}(i)^2
= \sum_s p_{ss}(i) = k_i.
\end{equation}

We now express $X_{i+1}$ using $P_i$,
%
\begin{equation}
\label{eqnformulaX} \qquad X_{i+1}=\frac{\|P_i \Ba_{i+1}\|^2-k_i}{k_i}= \frac{\sum_{s,t}p_{st}(i)a_sa_t-k_i}{k_i} :=
\sum_{s,t}q_{st}(i)a_sa_t-1,
\end{equation}
where $a_1=a_{i+1,1},\ldots,a_n=a_{i+1,n}$ are the coordinates of the
vector $\Ba_{i+1}$ and
\[
q_{st}(i):=\frac{p_{st}(i)}{k_i}.
\]
By \eqref{eqnP} we have $\sum_{s} q_{ss}(i) =1$ and $\sum_{s,t}
q_{st}(i)^2 = \frac{1}{k_i}. $

Because $\E a_s=0$ and $\E a_s^2=1$, and the $a_s$ are mutually
independent, we can show by using a routine calculation that [see
\eqref{eqnkey} from Section \ref{sectionmain2}]
%
\begin{equation}
\label{eqnvariance} \E\bigl(X^2_{i+1}|\CE_i
\bigr) =\frac{2}{k_i}-\sum_s
{q_{ss}(i)}^2\bigl(3-\E{a_s^4}
\bigr),
\end{equation}
where $\CE_i$ is the $\sigma$-algebra generated by the first $i$ rows
of $A_n$.

Define
\[
Y_{i+1} := -\frac{X_{i+1}^2} {2} +\frac{1}{k_i} -\frac{1}{2}
\sum_s {q_{ss}(i)}^2\bigl(3-
\E{a_s^4}\bigr)
\]
and
\[
Z_{i+1} := \frac{1}{2} \sum_s
{q_{ss}(i)}^2\bigl(3-\E{a_s^4}
\bigr).
\]
The reason we split $-\frac{X_{i+1}^2} {2}+\frac{1}{k_i}$ into the
sum of $Y_{i+1}$ and $Z_{i+1}$ is that $\E(Y_{i+1}|\CE_i)=0$ and its
variance can be easily computed.

\begin{lemma}\label{lemmamain22}
\[
\P \biggl(\biggl|\frac{\sum_{i< n_0} Y_{i+1}}{\sqrt{2\log n}}\biggr| \ge\log^{-1/3+o(1)} n \biggr) \le
\log^{-1/3+o(1)} n.
\]
\end{lemma}
To complete the proof of Lemma \ref{lemmamain2} from Lemma \ref
{lemmamain22}, it suffices to show that the sum of the $Z_i$ is negligible,
%
\begin{equation}
\label{eqnZ} \P \biggl(\frac{\sum_{i< n_0} Z_{i+1}}{\sqrt{2\log n}}=\Omega\biggl(\frac
{\log\log n}{\sqrt{2\log n}}\biggr)
\biggr)=O\bigl(n^{-100}\bigr).
\end{equation}
Our main technical tool will be the following lemma.

\begin{lemma}\label{lemmaerrorterm} With probability $1-O(n^{-100})$
we have
\[
\sum_{i<n_0}\sum_s
{q_{ss}(i)}^2= O(\log\log n).
\]
\end{lemma}

Noticing that $\E a_s^4$ is uniformly bounded (by condition \textup
{C0}), it follows that with probability $1- O(n^{-100})$,
\[
\sum_{i<n_0}\sum_s
q_{ss}(i)^2\bigl|3- \E a_s^4\bigr| =O(\log
\log n),
\]
proving \eqref{eqnZ}.

\section{\texorpdfstring{Proof of Lemmas \protect\ref{lemmamain1} and~\protect\ref{lemmamain2}: Mid game}
{Proof of Lemmas 2.6 and 2.7: Mid game}}\label{sectionstep2}

The key idea for proving Lemma \ref{lemmaerrorterm} is to establish a
good upper bound for $|q_{ss}(i)|$.
For this, we need some new tools. Our main ingredient is the following
delocalization result, which is a variant of
a result from \cite{TVhard} (see also \cite{E} and \cite{TVsurvey}
for recent surveys), asserting that with high probability all unit
vectors in the orthogonal complement of a random subspace with high
dimension have small infinity norm.

\begin{lemma}\label{lemmainfinitynorm} For any constant $\beta>0$
the following holds for all sufficiently large constant $\alpha>0$.
Assume that the components of $\Ba_1,\ldots, \Ba_{n_1}$, where
$n_1:=n-n\log^{-4\alpha}n$, are independent random variables of mean
zero, variance one and bounded in absolute value by $\log^{\beta} n$.
Then with probability $1-O(n^{-100})$, the following holds for all unit
vectors $\Bv$ of the space $V_{n_1}^\bot$:
\[
\|\Bv\|_\infty= O\bigl(\log^{-2\alpha}n\bigr).
\]
\end{lemma}

\begin{pf*}{Proof of Lemma \ref{lemmaerrorterm} assuming Lemma \ref
{lemmainfinitynorm}} Write
\begin{eqnarray*}
S&=&\sum_{i\le n_1} \sum_s
{q_{ss}(i)}^2 + \sum_{n_1 < i < n_0}
\sum_s {q_{ss}(i)}^2
\\
&:=&S_1+S_2.
\end{eqnarray*}
Note that as $q_{st}(i)=p_{st}(i)/k_i$,
\[
\sum_s q_{ss}(i)^2 \le\sum
_{s,t} q_{st}(i)^2 =\sum
_{s,t} \frac
{p_{st}(i)^2}{k_i^2} = \frac{1}{k_i}=
\frac{1}{(n-i)}.
\]
Hence,
\[
S_1\le\sum_{i\le n_1} \sum
_s {q_{ss}(i)}^2 \le \sum
_{i\le n_1} \frac{1}{(n-i)} =O(\log\log n).
\]
To bound $S_2$, note that
\[
p_{ss}(i)=\Be_s^T P_i
\Be_s = \|P_i\Be_s\|^2 = \bigr|
\langle\Be_s,\Bv \rangle\bigr|^2
\]
for some unit vector $\Bv\in V_i^\bot$.

Thus, if $i> n_1$, then $V_i^\bot\subset V_{n_1}^\bot$ and, hence, by
Lemma \ref{lemmainfinitynorm}
%
\begin{equation}
\label{eqndiagonal} p_{ss}(i) \le\|\Bv\|^2_\infty=
O\bigl(\log^{-4\alpha}n\bigr).
\end{equation}

It follows that
\begin{eqnarray*}
S_2&\le &\sum_{n_1 < i < n_0} \max_sp_{ss}(i)
\sum_s \frac{p_{ss}(i)}{(n-i)^2}
\\
&=& O\bigl(\log^{-4\alpha}n\bigr)\sum_{n_1 \le i < n_0}
\frac{1}{(n-i)}=O \bigl(\log^{-4\alpha+1 }n\bigr),
\end{eqnarray*}
completing the proof of Lemma \ref{lemmaerrorterm}.\vadjust{\goodbreak}
\end{pf*}

We now focus on the infinity norm of $\Bv$ and follow an argument from
\cite{TVhard}.

\begin{pf*}{Proof of Lemma \ref{lemmainfinitynorm}}
By the union bound, it suffices to show that $|v_1|=O(\log^{-2\alpha
}n)$ with probability at least $1-O(n^{-101})$, where $v_1$ is the
first coordinate of $\Bv$.

Let $B$ be the matrix formed by the first $n_1$ rows $\Ba_1,\ldots,
\Ba_{n_1}$ of $A$. Assume that $\Bv\in V_{n_1}^\bot$ is a unit
vector, then
\[
B\Bv=0.
\]

Let $\Bw$ be the first column of $B$, and $B'$ be the matrix obtained
by deleting $\Bw$ from $B$. Clearly,
%
\begin{equation}
\label{eqnlemmas1} v_1\Bw=-B'\Bv',
\end{equation}
where $\Bv'$ is the vector obtained from $\Bv$ by deleting $v_1$.

We next invoke the following result, which is a variant of \cite{TVhard}, Lemma
4.1.
This lemma was proved using a method of Guionet and Zeitouni \cite
{GZ}, based on Talagrand's inequality.

\begin{lemma}[(Concentration of singular values)]\label{lemmaMPlaw} For
any constant $\beta>0$ the following holds for all sufficiently large
constant $\alpha>0$.
Let $A_n$ be a random matrix of size $n$ by $n$, where the entries
$a_{ij}$ are independent random variables of mean zero, variance one
and bounded in absolute value by $\log^{\beta} n$. Then for any
$n/\log^\alpha n\le k\le n/2$, there exist $2k$ singular values of
$A_n$ in the interval $[0,ck/\sqrt{n}]$, for some absolute constant
$c$, with probability at least $1-O(n^{-101})$.
\end{lemma}

We can prove Lemma \ref{lemmaMPlaw} by following the arguments in
\cite{TVhard}, Lemma 4.1, almost word by word.

By the interlacing law and Lemma \ref{lemmaMPlaw}, we conclude that
$B'$ has $n-n_1$ singular values in the interval $[0,c(n-n_1)/\sqrt {n}]$ with probability $1-O(n^{-101})$.

Let $H$ be the space spanned by the left singular vectors of these
singular values, and let $\pi$ be the orthogonal projection onto $H$.
By definition, the spectral norm of $\pi B'$ is bounded,
\[
\bigl\|\pi B'\bigr\| \le c(n-n_1)/\sqrt{n}.
\]
Thus, \eqref{eqnlemmas1} implies that
%
\begin{equation}
\label{eqnlemmas2} |v_1| \|\pi\Bw\| \le c(n-n_1)/
\sqrt{n},
\end{equation}
here we used the fact that $\Bw$ is independent from $B'$, and thus
from $\pi$.

On the other hand, since the dimension of $H$ is $n-n_1$, Lemma \ref
{lemmaTalagrand} implies that $\|\pi\Bw\| \ge\sqrt{n-n_1}/2$ with
probability $1-4\exp(-(n-n_1)/16)=1-O(n^{-\omega(1)})$.

It thus follows from \eqref{eqnlemmas2} that
\[
|v_1|=O\bigl(\log^{-2\alpha} n\bigr).
\]
\upqed\end{pf*}

\section{\texorpdfstring{Proof of Lemma \protect\ref{lemmamain1}: End game}
{Proof of Lemma 2.6: End game}}\label
{sectionmain1}



Recall from \eqref{boundonX} that conditioned on any first $i$ rows,
$|X_i|=O(k_i^{-3/8})$ with probability $1-O(\exp(-\log^2n/2))$. So,
by paying an extra term of $O(\exp(-\log^2n/2))$ in probability, it
suffices to justify Lemma \ref{lemmamain1} for the sequence
$X_i':=X_i \cdot\BI_{|X_i|=O(k_i^{-3/8})}$.

On the other hand, the sequence $X_{i+1}'$ is not a martingale
difference sequence, so we slightly modify $X_{i+1}'$ to
$X_{i+1}'':=X_{i+1}' -\E(X_{i+1}'|\CE_i)$ and prove the claim for the
sequence $X_{i+1}'$, here we recall that $\CE_i$ is the $\sigma
$-algebra generated by the first $i$ rows of $A_n$. In order to show
that this modification has no effect whatsoever, we first demonstrate
that $\E(X_{i+1}'|\CE_i)$ is extremely small.

Recall from \eqref{eqnformulaX} that $X_{i+1}= \sum_{s,t}q_{st}(i)a_sa_t-1$. By the Cauchy--Schwarz inequality
and the assumption that $a_s$ are bounded in absolute value by $\log^{O(1)} n$, we have with probability one
%
\begin{eqnarray}
\label{boundX} |X_{i+1}|^2 &\le&\biggl(1+\sum
_{s,t}q_{st}(i)^2\biggr) \biggl(1+\sum
_{s,t}a_s^2a_t^2
\biggr) =(1+1/k_i) \biggl(1+\sum_{s,t}a_s^2a_t^2
\biggr)
\nonumber
\\[-8pt]
\\[-8pt]
\nonumber
&\le&2 \biggl(1+\sum_{s,t}a_s^2a_t^2
\biggr)\le n^2 \log^{O(1)}n.
\end{eqnarray}

Thus, with probability one
%
\begin{equation}
\label{eqnexpectation} \qquad\bigl|\E\bigl(X_{i+1}'|
\CE_i\bigr)\bigr| = \bigl|\E\bigl(X_{i+1}'|
\CE_i\bigr) - \E(X_{i+1}|\CE_i)\bigr| \le\exp
\bigl(-\bigl(\tfrac{1}{2}-o(1)\bigr)\log^2 n\bigr).
\end{equation}

To justify Lemma \ref{lemmamain1} for the sequence $X_{i+1}''$, we
apply Theorem \ref{theoremmartingale}.

The key point here is that thanks to the indicator function in the
definition of
$X_{i+1}'$ and the fact that the difference between $X_{i+1}^{\prime\prime}$ and
$X_{i+1}'$ is negligible,
$X_{i+1}''$ is bounded by $O(k_i^{-3/8})$ with probability one,
so the conditions $\E(|X_{i+1}''|^3|\CE_i)\le\gamma_i \E
({X_{i+1}''}^2|\CE_i)$ in Theorem
\ref{theoremmartingale} are satisfied with
\[
\gamma_i=O\bigl(k_i^{-3/8}\bigr)= O\bigl(
\log^{-3\alpha/8}n \bigr).
\]

We need to estimate $s_{n_0},v_{n_0}$ with respect to the sequence
$X_{i+1}''$. However, thanks to
the observations above, $X_{i+1}$ and $X_{i+1}^{\prime\prime}$ are very close,
and so it suffices to compute these values with respect to the sequence
$X_{i+1}$.

Recall from \eqref{eqnvariance} that
\[
\E\bigl(X^2_{i+1}|\CE_i\bigr) =
\frac{2}{k_i}-\sum_s {q_{ss}(i)}^2
\bigl(3-\E{a_s^4}\bigr).
\]

Also, recall from Section \ref{sectionstep2} that with probability
$1-O(n^{-100})$,
\[
\sum_{i < n_0} \sum_s
q_{ss}(i)^2 \bigl(3 - \E a_s^4
\bigr) =O(\log\log n).
\]

This bound, together with \eqref{eqnvariance} and \eqref{boundX},
imply that with probability one
%
\[
\E\biggl(\sum_{i < n_0} X_{i+1}^2|
\CE_i \biggr) = \sum_{i < n_0}
\frac{2} {
k_i} +O(\log\log n) =2\log n +O(\log\log n),
\]
which in turn implies that $v_{n_0}^2 = 2 \log n +O(\log\log n)$ with
probability $1 - O(n^{-100})$.

Using \eqref{boundX} again, because $n^{-100} n^2 \log^{O(1)} n
=o(1)$, we deduce that
%
\begin{equation}
\label{boundsm} s_{n_0}^2 = 2\log n + O(\log\log n).
\end{equation}

With another application of \eqref{boundX}, we obtain
\[
\E\biggl|\frac{v_{n_0}^2}{s_{n_0}^2}-1\biggr| \le O\biggl(\frac{\log\log n}{\log n} \biggr) +
n^{-100} n^2 \log^{O(1)} n .
\]

It follows that
\[
\E^{1/3} \biggl|\frac{v_{n_0}^2}{s_{n_0}^2}-1\biggr| \le\log^{-1/3+o(1)}n.
\]

By the conclusion of Theorem \ref{theoremmartingale} and setting
$\alpha$ sufficiently large, we
conclude
\begin{eqnarray*}
&&\sup_{x\in\R}\biggl |\P\biggl(\frac{\sum_{i<n_0}X_{i+1}''}{s_{n_0}} <x\biggr) - \Phi(x) \biggr| \\
&&\qquad\le L
\biggl(\frac{\log^{-3\alpha/8}n \times\log
n_0}{s_{n_0}} + \E^{1/3}\biggl(\biggl|\frac{v_{n_0}^2}{s_{n_0}^2}-1\biggr|\biggr)
\biggr)
\\
&&\qquad\le \log^{-1/3+o(1)} n,
\end{eqnarray*}
completing the proof of Lemma \ref{lemmamain1}.


\section{\texorpdfstring{Proof of Lemma \protect\ref{lemmamain2}: End game}
{Proof of Lemma 2.7: End game}}\label
{sectionmain2}

Our goal is to justify Lemma \ref{lemmamain22}, which together with
\eqref{eqnZ} verify
Lemma \ref{lemmamain2}.

We will show that the variance $\Var(\sum_{i<n_0}Y_{i+1})$ is small
and then use Chebyshev's inequality.
The proof is based on a series of routine, but somewhat tedious
calculations. We first show that the expectations of the $Y_{i+1}$'s
are zero, and so are the covariances $\E(Y_{i+1} Y_{j+1})$ by
an elementary manipulation. The variances $\Var(Y_{i+1})$ will be
bounded from above by the Cauchy--Schwarz inequality.

We start with the formula $X_{i+1}^2=(\sum_{s,t}q_{st}(i)a_sa_t)^2 -2
\sum_{s,t}q_{st}(i)a_sa_t +1$. Observe that
\begin{eqnarray*}
\biggl(\sum_{s,t}q_{st}(i)a_sa_t
\biggr)^2&=&\biggl(\sum_{s}q_{ss}(i)a_{s}^2
+ \sum_{s\neq t}q_{st}(i)a_sa_t
\biggr)^2
\\
&=&\biggl(\sum_{s}q_{ss}(i){a_s}^2
\biggr)^2+\biggl(\sum_{s\neq t}q_{st}(i)a_sa_t
\biggr)^2\\
&&{}+ 2 \biggl(\sum_{s}q_{ss}(i)a_{s}^2
\biggr) \biggl(\sum_{s\neq t}q_{st}(i)a_sa_t
\biggr).
\end{eqnarray*}

Expanding each term, using the fact that $\sum_{s}q_{ss}(i)=1$ and
$\sum_{s,t}q_{st}(i)^2 = \frac{1}{k_i}$, we have
\begin{eqnarray*}
&&\biggl(\sum_{s}q_{ss}(i){a_s}^2
\biggr)^2\\
&&\qquad=\biggl(\sum_{s}q_{ss}(i)
\biggr)^2-\sum_sq_{ss}(i)^2
\bigl(1-a_s^4\bigr)+2\sum_{s\neq
t}q_{ss}(i)q_{tt}(i)
\bigl(a_s^2a_t^2-1\bigr)
\\
&&\qquad =1-\sum_sq_{ss}(i)^2
\bigl(1-a_s^4\bigr)+2\sum_{s\neq
t}q_{ss}(i)q_{tt}(i)
\bigl(a_s^2a_t^2-1\bigr)
\end{eqnarray*}
and
\begin{eqnarray*}
\biggl(\sum_{s\neq t}q_{st}(i)a_sa_t
\biggr)^2 &=&2\sum_{s\neq t}q_{st}(i)^2
+ 2\sum_{s\neq t}q_{st}(i)^2
\bigl(a_s^2a_t^2-1\bigr)\\
&&{}+ 2\mathop{\sum_{s_1\neq t_1,s_2\neq t_2 }}_{\{s_1,t_1\}\neq\{s_2,t_2\}} q_{s_1t_1}(i)q_{s_2t_2}(i)a_{s_1}a_{t_1}a_{s_2}a_{t_2}
\\
&=&\frac{2}{k_i} - 2\sum_{s}
q_{ss}(i)^2 + 2\sum_{s\neq t}q_{st}(i)^2
\bigl(a_s^2a_t^2-1\bigr) \\
&&{}+ 2
\mathop{\sum_{s_1\neq t_1,s_2\neq t_2 }}_{\{s_1,t_1\}\neq\{s_2,t_2\}} q_{s_1t_1}(i)q_{s_2t_2}(i)a_{s_1}a_{t_1}a_{s_2}a_{t_2}
)
\end{eqnarray*}
as well as
\begin{eqnarray*}
&&2 \biggl(\sum_{s}q_{ss}(i)a_{s}^2
\biggr) \biggl(\sum_{s\neq t}q_{st}(i)a_sa_t
\biggr)\\
&&\qquad= 2\biggl(\sum_{s}q_{ss}(i)
\bigl(a_{s}^2-1\bigr)\biggr) \biggl(\sum
_{s\neq t}q_{st}(i)a_sa_t
\biggr)+2 \sum_{s\neq t}q_{st}(i)a_sa_t
.
\end{eqnarray*}

It follows that
%
\begin{eqnarray}
\label{eqnkey} -2Y_{i+1}&=& X_{i+1}^2 -
\frac{2}{k_i} + \sum_s {q_{ss}(i)}^2
\bigl(3-\E {a_s^4}\bigr)\nonumber\\
&=&\biggl(\sum
_{s,t} q_{st}(i)a_sa_t-1
\biggr)^2 - \frac{2}{k_i} + \sum_s
{q_{ss}(i)}^2\bigl(3-\E{a_s^4}
\bigr)
\nonumber\\
&=&\biggl(\sum_{s,t}q_{st}(i)a_sa_t
\biggr)^2-1-2\sum_{s}q_{ss}(i)
\bigl(a_s^2-1\bigr)-2\sum_{s\neq t}q_{st}(i)a_sa_t
- \frac{2}{k_i}\nonumber\\
&&{} + \sum_s
{q_{ss}(i)}^2\bigl(3-\E {a_s^4}
\bigr)
\nonumber
\\[-8pt]
\\[-8pt]
\nonumber
&=&-2\sum_{s}q_{ss}(i)
\bigl(a_s^2-1\bigr) + \sum_s
{q_{ss}(i)}^2\bigl(a_s^4-\E
a_s^4\bigr)\\
&&{}+2\sum_{s\neq t}q_{ss}(i)q_{tt}(i)
\bigl(a_s^2a_t^2-1\bigr)+ 2\sum
_{s\neq
t}q_{st}(i)^2
\bigl(a_s^2a_t^2-1\bigr)
\nonumber
\\
&&{}+ 2\mathop{\sum_{s_1\neq t_1,s_2\neq t_2 }}_{\{
s_1,t_1\}\neq\{s_2,t_2\}} q_{s_1t_1}(i)q_{s_2t_2}(i)a_{s_1}a_{t_1}a_{s_2}a_{t_2}
\nonumber\\
&&{}+ 2 \biggl(\sum_{s}q_{ss}(i)
\bigl(a_{s}^2-1\bigr)\biggr) \biggl(\sum
_{s\neq t}q_{st}(i)a_sa_t
\biggr).\nonumber
\end{eqnarray}

As $\E a_s=0, \E a_s^2=1$, and the $a_s$'s are mutually independent
with each other and with every row of index at most $i$ [and in
particular with $q_{st}(i)$'s], every term in the last formula is zero,
and so we infer that $\E(Y_{i+1})=0$ and $\E(Y_{i+1}|\CE_i)=0$,
confirming \eqref{eqnvariance}. With the same reasoning, we
can also infer that the covariance $\E(Y_{i+1}Y_{j+1})=0$ for all $j<i$.

It is thus enough to work with the diagonal terms $\Var(Y_{i+1})$. We
have
\begin{eqnarray*}
\Var(Y_{i+1})&=&\E \biggl[-\sum_{s}q_{ss}(i)
\bigl(a_s^2-1\bigr) +\frac{1}{2} \sum
_s {q_{ss}(i)}^2\bigl(a_s^4-
\E a_s^4\bigr)
\\
&&\hspace*{14pt}{}+\sum_{s\neq t}q_{ss}(i)q_{tt}(i)
\bigl(a_s^2a_t^2-1\bigr)+ \sum
_{s\neq
t}q_{st}(i)^2
\bigl(a_s^2a_t^2-1\bigr)
\\
&&\hspace*{14pt}{}+ \mathop{\sum_{ s_1\neq t_1,s_2\neq t_2}}_{\{
s_1,t_1\}\neq\{s_2,t_2\}} q_{s_1t_1}(i)q_{s_2t_2}(i)a_{s_1}a_{t_1}a_{s_2}a_{t_2}\\
&&\hspace*{70pt}{}+
\biggl(\sum_{s}q_{ss}(i)
\bigl(a_{s}^2-1\bigr)\biggr) \biggl(\sum
_{s\neq t}q_{st}(i)a_sa_t
\biggr) \biggr]^2.
\end{eqnarray*}

After a series of cancellations, and because of condition \textup{C0},
we have
\begin{eqnarray*}
\Var(Y_{i+1}) &\le& O\biggl (\E \biggl[\sum_s
q_{ss}(i)^2+ \sum_s
q_{ss}(i)^4 +\sum_{s \neq t_1, s\neq t_2}q_{ss}(i)^2
q_{t_1t_1}(i)q_{t_2t_2}(i)
\\
&&\hspace*{29pt}{}+\sum_{s\neq t_1, s\neq t_2} q_{st_1}(i)^2q_{st_2}(i)^2\\
&&\hspace*{29pt}{}+
\sum_{s_1
\neq t_1,s_2 \neq
t_2}\bigl|q_{s_1t_1}(i)q_{s_1t_2}(i)q_{s_2t_1}(i)q_{s_2t_2}(i)\bigr|\\
&&\hspace*{29pt}{}+
\sum_{s,t}q_{ss}(i)q_{tt}(i)q_{st}(i)^2
\\
&&\hspace*{29pt}{}+\sum_sq_{ss}(i)^3+ \sum
_{s,t}q_{ss}(i)^2
q_{tt}(i)+ \sum_{s,t}q_{ss}(i)
q_{st}(i)^2\\
&&\hspace*{29pt}{}+\sum_{s,t}\bigl|q_{ss}(i)q_{tt}(i)q_{st}(i)\bigr|
\\
&&\hspace*{29pt}{}+\sum_{s,t}q_{ss}(i)^3q_{tt}(i)
+\sum_{s\neq t} q_{ss}(i)^2
q_{st}(i)^2\\
&&\hspace*{29pt}{}+\sum_{s,t}\bigl|q_{ss}(i)^2q_{tt}(i)q_{st}(i)\bigr|
\\
&&\hspace*{29pt}{}+\sum_{s,t}q_{ss}(i)q_{tt}(i)q_{st}(i)^2+
\sum_{s\neq t}\bigl|q_{ss}(i)^2
q_{tt}(i)q_{st}(i)\bigr|
\\
&&\hspace*{29pt}{}+\sum_{s,t}\bigl|q_{st}(i)^3q_{ss}(i)\bigr|
\\
&&\hspace*{85pt}{}+\sum_{s \neq t_1,s\neq t_2, t_1\neq
t_2}\bigl|q_{ss}(i)q_{st_1}(i)q_{st_2}(i)q_{t_1t_2}(i)\bigr|
\biggr] \biggr),
\end{eqnarray*}
where the first two rows consist of the squares of the terms appearing
in $Y_{i+1}$ (after deleting several sums of zero expected value), and
each of the following rows was obtained by expanding the product of
each term with the rest in the order of their appearance.

Because $\sum_{s,t}q_{st}(i)^2 = \frac{1}{k_i}$, one has $\max_{s,t}|q_{st}(i)|\le\frac{1}{\sqrt{k_i}}$ for all $s,t$. Recall
furthermore that $\sum_{s}q_{ss}(i)=1$ and $0\le q_{ss}(i)$ for all
$s$. We next estimate the terms under consideration one by one as follows.

First, the sums $ \sum_s q_{ss}^3(i), \sum_s q_{ss}(i)^4$, $\sum_{s,t}q_{ss}(i) q_{st}(i)^2, \sum_{s,t} q_{ss}(i)^2 q_{st}(i)^2$,
$\sum_{s,t}q_{ss}(i)q_{tt}(i)q_{st}(i)^2$, and $\sum_{s,t}|q_{st}(i)^3q_{ss}(i)|$ can be bounded by\break  $\max_{s,t}|q_{st}(i)|\sum_{s,t} q_{st}^2(i)$, and so by $k_i^{-3/2}$.

Second, by applying the Cauchy--Schwarz inequality if needed, one can
bound the sums $\sum_{s,t_1,t_2} q_{st_1}(i)^2q_{st_2}(i)^2$, $\sum_{s_1,t_1,s_2,t_2}|q_{s_1t_1}(i)q_{s_1t_2}(i)q_{s_2t_1}(i)q_{s_2t_2}(i)|$,
and $\sum_{s,t_1,t_2}|q_{ss}(i)q_{st_1}(i)q_{st_2}(i)q_{t_1t_2}(i)|$
by $2 (\sum_{s,t}q_{st}^2(i))^2$, and so by $2k_i^{-2}$.

We bound the remaining terms as follows:

\begin{itemize}
\item$\sum_{s,t_1, t_2}q_{ss}(i)^2 q_{t_1t_1}(i)q_{t_2t_2}(i)= (\sum_sq_{ss}(i)^2) (\sum_t q_{tt}(i))^2 = \sum_sq_{ss}(i)^2.$
\item$\sum_{s,t}q_{ss}(i)^2 q_{tt}(i)+ \sum_{s,t}q_{ss}(i)^3q_{tt}(i) \le2 (\sum_sq_{ss}(i)^2)(\sum_{t}q_{tt}(i))= 2 \sum_sq_{ss}(i)^2 .$
\item$\sum_{s,t}|q_{ss}(i)q_{tt}(i)q_{st}(i)|\le\sum_{s,t}q_{ss}(i)(q_{tt}(i)^2 + q_{st}(i)^2)
\le\sum_{t}q_{tt}(i)^2+\break
\max_{s}q_{ss}(i)\sum_{s,t}q_{st}(i)^2 $ $\le \sum_{t}q_{tt}(i)^2 +
k_i^{-3/2}.$
\item$\sum_{s,t}|q_{ss}(i)^2 q_{tt}(i)q_{st}(i)| \le\sup_{s,t}|q_{st}(i)|\sum_{s,t}q_{ss}(i)^2q_{tt}(i) \le\sum_{s}q_{ss}(i)^2/\sqrt{k_i}$.
\end{itemize}

Putting all bounds together, we have
%
\begin{eqnarray}
\label{eqnlast} \Var\biggl(\sum_{i<n_0}Y_{i+1}
\biggr)&=&\sum_{i<n_0}\Var(Y_{i+1})=O \biggl(\E
\biggl(\sum_{i<n_0} \sum_sq_{ss}(i)^2
+ \sum_{i<n_0}k_i^{-3/2}\biggr)
\biggr)
\nonumber
\\[-8pt]
\\[-8pt]
\nonumber
&=&O(\log\log n),
\end{eqnarray}
where we applied Lemma \ref{lemmaerrorterm} in the last estimate.

To complete the proof, we note from the estimate of $s_{n_0}^2$ of
Section \ref{sectionmain1} and from Lemma \ref{lemmaerrorterm} that
$|\sum_{i<n_0}\E Y_{i+1}|=O(\log\log n)$. Thus, by Chebyshev's inequality
\[
\P \biggl(\biggl|\frac{\sum_{i<n_0}Y_{i+1}}{\sqrt{2\log n}}\biggr| \ge\log^{-1/3+o(1)} n \biggr) =
\log^{-1/3+o(1)}n.
\]

\section{\texorpdfstring{Proof of Lemma \protect\ref{lemmamain3}}
{Proof of Lemma 2.8}}\label{sectionmain3}



We recall that, with $i\ge n_0$, $\Delta_{i+1}^2$ is a Chi-square
random variable of degree $n-i$. Let us first consider the lower tail;
it suffices to show
%
\begin{equation}
\label{eqnlargeindex}\quad  \P \biggl(\sum_{n_0\le i}
\frac{\log({\Delta_{i+1}^2}/{(n-i)})}{\sqrt{2\log n}}< - \log^{-1/2 +c} n \biggr) =o\bigl( \exp \bigl(-
\log^{c/2} n\bigr)\bigr)
\end{equation}
for any constant $0<c<1/100$.

By properties of the normal distribution, it is easy to show that
$\Delta_n^2$ and $\Delta_{n-1}^2$ are at least $\exp(-\frac{\sqrt {2}}{4}\log^{c} n)$ with probability $1-\exp(-\Omega(\log^c n ))$,
so we can omit these terms from the sum. It now suffices to show that
%
\begin{eqnarray}
\label{eqnlargeindex2}&& \P \biggl(\sum_{n_0\le i \le n-3}
\frac{\log({\Delta_{i+1}^2}/{(n-i)})}{\sqrt{2\log n}}< - \frac{1}{2} \log^{-1/2 +c} n \biggr)
\nonumber
\\[-8pt]
\\[-8pt]
\nonumber
&&\qquad=o\bigl(
\exp\bigl(-\log^{c/2} n\bigr)\bigr)
\end{eqnarray}
for any small constant $0<c<1/100$.

Flipping the inequality inside the probability (by changing the sign of
the RHS and swapping the denominators and numerators in the logarithms
of the LHS) and using the Laplace transform trick (based on the fact
that the $\Delta^2_i$ are independent), we see that the probability in
question is at most
\[
\frac{ \E\prod_{i=n_0}^{n-3} {(n-i)}/{ \Delta_{i+1}^2 } }{ \exp
( ({1}/{\sqrt{2}}) \log^c n) }= \frac{ \prod_{i=n_0}^{n-3} \E {(n-i)}/{ \Delta_{i+1}^2 } }{ \exp
( ({1}/{\sqrt{2}}) \log^c n )}.
\]

Recall that $\Delta_{i+1}^2$ is a Chi-square random variable with
degree of freedom $n-i$, so
$\E\frac{1}{ \Delta^2_{i+1} } = \frac{1}{ n-i-2}$. Therefore, the
numerator in the previous
formula is $\frac{(n-n_0)(n-n_0-1)}{2} \le\log^{2 \alpha} n $.

Because
\[
\frac{\log^{2 \alpha} n }{ \exp( ({1}/{\sqrt{2}}) \log^c n )} =o\bigl(\exp\bigl(-\log^{c/2} n\bigr)\bigr),
\]
the desired bound follows.

The proof for the upper tail is similar (in fact simpler as we do not
need to treat the first two terms
separately) and we omit the details.

\section{\texorpdfstring{Deduction of Theorem \protect\ref{theoremmain1} from Theorem \protect\ref{theoremmainweak}}
{Deduction of Theorem 2.1 from Theorem 2.2}}\label{sectiondeduction}

Our plan is to replace one by one the last $n-n_0$
Gaussian rows of $A_n$ by vectors of components having zero mean, unit
variance and satisfying condition \textup{C0}. Our key tool here is
the classical Berry--Eseen inequality. In order to apply this lemma, we
will make
a crucial use of Lemma \ref{lemmainfinitynorm}.

\begin{lemma}[(\cite{B}, Berry--Esseen inequality)] \label{lemmaB-E}
Assume that $\Bv=(v_1,\ldots,v_n)$ is a unit vector. Assume that
$b_1,\ldots, b_n$ are independent random variables of mean zero,
variance one and satisfying condition \textup{C0}. Then we have
\[
\sup_x\bigl|\P(v_1b_1+\cdots+v_nb_n
\le x)-\Phi(x)\bigr|\le c \|\Bv\|_\infty,
\]
where $c$ is an absolute constant depending on the parameters appearing
in~\eqref{eqntailbound}.
\end{lemma}

We remark that in the original setting of Berry and Esseen, it suffices
to assume the finite third moment.

In application, $\Bv$ plays the role of the normal vector of the
hyperplane spanned by the remaining $n-1$ rows of $A$, and $\Delta_n=|v_1b_1+\cdots+v_nb_n|$, where $(b_1,\ldots,b_n)=\Bb$ is the vector
to be replaced.

For the deduction, it is enough to show the following.

\begin{lemma}\label{lemmareplacement} Let $A_n$ be a random matrix with
atom variables satisfying condition \textup{C0} and
nonsingular with probability one. Assume furthermore that $A_n$ has at
least one and at most
$\log^{\alpha} n$ Gaussian rows. Let $B_n$ be the random matrix
obtained from $A_n$ by replacing a Gaussian row vector $\Ba$ of $A_n$
by a random vector $\Bb=(b_1,\ldots,b_n)$ whose coordinates are
independent atom variables satisfying condition \textup{C0} such that
the resulting matrix is nonsingular with probability one. Then
\begin{eqnarray}\label{eqnreplacement}
&&\sup_x \biggl|\P_{B_n} \biggl(\frac{\log(\det B_n^2)-\log
(n-1)!}{\sqrt{2\log n}} \le x \biggr) \nonumber\\
&&\hspace*{7pt}\quad {}-
\P_{A_n} \biggl(\frac{\log
(\det{A_n}^2)-\log(n-1)!}{\sqrt{2\log n}}\le x \biggr) \biggr|
\\
&&\qquad \le O\bigl(\log^{-2\alpha}n\bigr).\nonumber
\end{eqnarray}
\end{lemma}

Clearly, Theorem \ref{theoremmain} follows from Theorem \ref
{theoremmainweak} by applying Lemma \ref{lemmareplacement} $\log^{\alpha}n$ times.

\begin{pf*}{Proof of Lemma \ref{lemmareplacement}}
Without loss of generality, we can assume that $B_n$ is obtained from
$A_n$ by replacing the last row $\Ba_n$. As $A_n$ is nonsingular,
$\dim(V_{n-1})=n-1$.

By Lemma \ref{lemmainfinitynorm}, by paying an extra term of
$O(n^{-100})$ in probability (which will be absorbed by the eventual
bound $\log^{-2\alpha}n$), we may also assume that the normal vector
$\Bv$ of $V_{n-1}$ satisfies
\[
\|\Bv\|_\infty=O\bigl(\log^{-2\alpha} n\bigr).
\]

Next, observe that
\[
\frac{\log(\det A^2)-\log(n-1)!}{\sqrt{2\log n}} = \frac{\sum_{i=0}^{n-2} \log(\Delta_{i+1}^2/(n-i)) + \log n}{\sqrt{2\log n}} +
\frac{\log\Delta_n^2}{\sqrt{2\log n}}
\]
and
\[
\frac{\log(\det{B}^2)-\log(n-1)!}{\sqrt{2\log n}} = \frac{\sum_{i=0}^{n-2} \log(\Delta_{i+1}^2/(n-i)) + \log n}{\sqrt{2\log n}} + \frac{\log{\Delta_n'}^2}{\sqrt{2\log n}},
\]
where $\Delta_n$ and $\Delta'_n$ are the distance from $\Ba_n$ and
$\Bb_n$ to $V_{n-1}$, respectively.

By Lemma \ref{lemmaB-E}, it is yielded that
\[
\sup_x\bigl|\P_{\Ba_n}\bigl(\Delta_n^2
\le x\bigr)-\P_{\Bb_{n}}\bigl({\Delta_n'}^2
\le x\bigr)\bigr|\le c\|\Bv\|_\infty= O\bigl(\log^{-2\alpha}n\bigr).
\]

Hence,
\begin{eqnarray*}
&&\sup_x \biggl|\P_{\Ba_n} \biggl(\frac{\log(\det A^2)-\log(n-1)!}{\sqrt {2\log n}} \le x \biggr) \\
&&\hspace*{-4pt}\qquad{}-
\P_{\Bb_n} \biggl(\frac{\log(\det
{B}^2)-\log(n-1)!}{\sqrt{2\log n}} \le x \biggr)\biggr|
\\
&&\hspace*{-4pt}\qquad= O\bigl(\log^{-2\alpha}n\bigr),
\end{eqnarray*}
completing the proof of Lemma \ref{lemmareplacement}.
\end{pf*}

\begin{appendix}
\section*{\texorpdfstring{Appendix: Simplifying the model: Deducing Theorem \lowercase{\protect\ref{theoremmain}} from Theorem \lowercase{\protect\ref{theoremmain1}}}
{Appendix: Simplifying the model: Deducing Theorem 1.1 from Theorem 2.1}}
\label{appendixmodel}

In this section we show that the two extra assumptions that $|a_{ij}|
\le\log^{\beta} n $ and~$A_n$ has full rank with probability one do
not violate the generality of Theorem \ref{theoremmain}.

To start with, we need a very weak lower bound on $|\det A_n|$.

\begin{lemma} \label{lemmalower} There is a constant $C$ such that
\[
\P\bigl(|\det A_n | \le n^{-C n}\bigr)\le n^{-1}.
\]
\end{lemma}

\begin{pf}
It follows from
\cite{TVcir}, Theorem 2.1, that there is a constant $C$ such that $\P(\sigma_n(A_n)\le n^{-C})\le n^{-1}$. Since $|\det A_n|$ is the product of its
singular values, the bound follows.
\end{pf}

\begin{remark}
The above bound is extremely weak. By modifying the proof in~\cite
{TVdet}, one can actually prove the
Tao--Vu lower bound \eqref{TVlow} for random matrices satisfying
\textup{C0}.
Also, sharper bounds on the least singular value are obtained in \cite
{TVsmooth,RV}.
However, for the arguments
in this section, we only need the bound on Lemma~\ref{lemmalower}.
\end{remark}

Let us start with the assumption $|a_{ij}| \le\log^{\beta} n$. We
can achieve this assumption using the standard truncation method (see
\cite{BS}
or \cite{TVlocal}). In what follows, we sketch the idea.

Notice that by condition \textup{C0}, we have, with probability at
least $1 -\exp\times ( -\log^{10} n)$, that all entries of $A_n$ have
absolute value at most $\log^{\beta} n$, for some constant $\beta>0
$ which may depend on the constants in \textup{C0}.

We replace the variable $a_{ij}$ by the variable $a_{ij}' := a_{ij} \I_{|a_{ij}| \le\log^{\beta} n }$, for all
$1 \le i, j \le n$ and let $A_n'$ be the random matrix formed by
$a_{ij}'$. Since with probability
at least $1-\exp(-\log^{10} n)$, $A_n =A_n'$, it is easy to show that if
$A_n'$ satisfies the claim of Theorem \ref{theoremmain}, then so does $A_n$.

While the entries of $A_n'$ are bounded by $\log^{\beta}n$, there is
still one problem we need to address,
namely, that the new variables $a_{ij}'$ do not have mean 0 and
variance one. We can achieve this by a simple normalization trick.
First observe that by property \textup{C0}, taking $\beta$
sufficiently large, it is easy to show that $\mu_{ij} = \E a_{ij}'$
has absolute value at most $n^{-\omega(1)}$ and $| 1- \sigma_{ij}|
\le n^{-\omega(1)}$, where $\sigma_{ij} $ is
the standard deviation of $a'_{ij}$. Now define
\[
a_{ij}^{\prime\prime}:=a_{ij}'-
\mu_{ij}
\]
and
\[
a_{ij}^{\prime\prime\prime} = \frac{a_{ij}^{\prime\prime}} {\sigma_{ij}} .
\]

Note that $a_{ij}^{\prime\prime\prime}$ now does have mean zero and variance one. Let
$A_n^{\prime\prime}$ and $A_n^{\prime\prime\prime}$ be the corresponding matrices of $a_{ij}''$
and $a_{ij}'''$, respectively.

By the Brun--Minkowski inequality we have
\[
\bigl|\det\bigl(A_n'\bigr)\bigr| \le\bigl(\bigl|\det
A_n^{\prime\prime} \bigr|^{1/n} + |\det N_n|^{1/n}
\bigr)^n,
\]
where $N_n$ is the matrix formed by $\mu_{ij}$.

Since $|\mu_{ij}| = n^{-\omega(1)}$, by Hadamard's bound $|\det
N_n|^{1/n} \le n^{-\omega(1)}$. On the other hand, we have by Lemma
\ref{lemmalower} that $\P(|\det A_n^{\prime\prime}|^{1/n} \ge n^{-C})\ge
1-n^{-1}$. It thus follows that
\[
\P\bigl(\bigl|\det A_n'\bigr| \le\bigl(1+o(1)\bigr) \bigl|\det
A_n^{\prime\prime}\bigr|\bigr)\ge1-n^{-1}.
\]
We can prove a matching lower bound by the same argument. From here, we
conclude that if
$|\det A_n^{\prime\prime} |$ satisfies the conclusion of Theorem \ref
{theoremmain}, then so does $|\det A_n'|$.

To pass from $\det(A_n^{\prime\prime})$ to $\det(A_n^{\prime\prime\prime})$, we apply the
Brunn--Minkowski inequality again,
\[
\bigl|\det\bigl(A_n^{\prime\prime\prime}\bigr)\bigr| \le\bigl(\bigl|\det
A_n^{\prime\prime} \bigr|^{1/n} + \bigl|\det N_n'\bigr|^{1/n}
\bigr)^n,
\]
where $N_n'$ is the matrix form by $a_{ij}''(1-\sigma_{ij}^{-1})$.
Noting that $|1- \sigma_{ij}^{-1}| \le n^{-\omega(1)}$ and
$|a_{ij}''|=\log^{O(1)}n$, we infer that $|\det(A_n^{\prime\prime})|$ and
$|\det(A_n^{\prime\prime\prime})|$ are comparable with high probability
\[
\P\bigl(\bigl|\det A_n^{\prime\prime}\bigr| = \bigl(1+o(1)\bigr) \bigl|\det
A_n^{\prime\prime\prime}\bigr|\bigr)\ge1-n^{-1}.
\]

Now we address the assumption that $A_n$ has full rank with probability one.
Notice that this is usually not true when the $a_{ij}$ have discrete
distribution (such as Bernoulli).
However, we find the following simple trick that makes the assumption
valid for our study.

Instead of the entry $a_{ij}$, consider $a_{ij}' := (1 -\eps^2)^{1/2}
a_{ij} + \eps\xi_0 $ where $\xi_0$ is
uniform on the interval $[-1,1]$ and $\eps$ is very small, say,
$n^{-1000 n}$. It is clear that the matrix~$A_n'$ formed by the
$a_{ij}'$ has full rank with probability one. On the other hand, it is
easy to show that by the Brunn--Minkowski inequality and Hadamard's bound
\[
|\det A_n| = \bigl(\bigl|\det A_n'\bigr|^{1/n}
\pm O\bigl(n^{-500}\bigr)\bigr)^n .
\]
Furthermore, by Lemma \ref{lemmalower}, $|\det A_n| \ge n^{-Cn}$ with
probability $1-n^{-1}$, and so we can conclude as in the previous argument.
\end{appendix}
%



\printaddresses


\begin{thebibliography}{30}

\bibitem{BS}
%
\begin{bbook}[auto:STB|2013/01/29|08:09:18]
\bauthor{\bsnm{Bai},~\bfnm{Z.}\binits{Z.}} \AND
\bauthor{\bsnm{Silverstein},~\bfnm{J.}\binits{J.}}
(\byear{2006}).
\btitle{Spectral Analysis of Large Dimensional Random Matrices}.
\bpublisher{Science press}, \blocation{Beijing}.
\bptok{imsref}%
\end{bbook}
%
\endbibitem

\bibitem{B}
%
\begin{barticle}[mr]
\bauthor{\bsnm{Berry},~\bfnm{Andrew~C.}\binits{A.~C.}}
(\byear{1941}).
\btitle{The accuracy of the {G}aussian approximation to the sum of independent
variates}.
\bjournal{Trans. Amer. Math. Soc.}
\bvolume{49}
\bpages{122--136}.
\bid{issn={0002-9947}, mr={0003498}}
\bptok{imsref}%
\end{barticle}
%
\endbibitem

\bibitem{BVW}
%
\begin{barticle}[mr]
\bauthor{\bsnm{Bourgain},~\bfnm{Jean}\binits{J.}},
\bauthor{\bsnm{Vu},~\bfnm{Van~H.}\binits{V.~H.}} \AND
\bauthor{\bsnm{Wood},~\bfnm{Philip~Matchett}\binits{P.~M.}}
(\byear{2010}).
\btitle{On the singularity probability of discrete random matrices}.
\bjournal{J. Funct. Anal.}
\bvolume{258}
\bpages{559--603}.
\bid{doi={10.1016/j.jfa.2009.04.016}, issn={0022-1236}, mr={2557947}}
\bptok{imsref}%
\end{barticle}
%
\endbibitem

\bibitem{Dembo}
%
\begin{barticle}[mr]
\bauthor{\bsnm{Dembo},~\bfnm{A.}\binits{A.}}
(\byear{1989}).
\btitle{On random determinants}.
\bjournal{Quart. Appl. Math.}
\bvolume{47}
\bpages{185--195}.
\bid{issn={0033-569X}, mr={0998095}}
\bptok{imsref}%
\end{barticle}
%
\endbibitem

\bibitem{MO}
%
\begin{barticle}[mr]
\bauthor{\bsnm{El~Machkouri},~\bfnm{M.}\binits{M.}} \AND
\bauthor{\bsnm{Ouchti},~\bfnm{L.}\binits{L.}}
(\byear{2007}).
\btitle{Exact convergence rates in the central limit theorem for a
class of
martingales}.
\bjournal{Bernoulli}
\bvolume{13}
\bpages{981--999}.
\bid{doi={10.3150/07-BEJ6116}, issn={1350-7265}, mr={2364223}}
\bptok{imsref}%
\end{barticle}
%
\endbibitem

\bibitem{E}
%
\begin{bmisc}[auto:STB|2013/01/29|08:09:18]
\bauthor{\bsnm{Erd{\H{o}}s},~\bfnm{L.}\binits{L.}}
\bhowpublished{Universality of Wigner random matrices: A survey of recent
results. Available at arXiv:\arxivurl{1004.0861v2}.}
\bptok{imsref}%
\end{bmisc}
%
\endbibitem

\bibitem{FT}
%
\begin{barticle}[auto:STB|2013/01/29|08:09:18]
\bauthor{\bsnm{Forsythe},~\bfnm{G.~E.}\binits{G.~E.}} \AND
\bauthor{\bsnm{Tukey},~\bfnm{J.~W.}\binits{J.~W.}}
(\byear{1952}).
\btitle{The extent of n random unit vectors}.
\bjournal{Bull. Amer. Math. Sot.}
\bvolume{58}
\bpages{502}.
\bptok{imsref}%
\end{barticle}
%
\endbibitem

\bibitem{G2}
%
\begin{barticle}[mr]
\bauthor{\bsnm{G{\={\i}}rko},~\bfnm{V.~L.}\binits{V.~L.}}
(\byear{1979}).
\btitle{A central limit theorem for random determinants}.
\bjournal{Theory Probab. Appl.}
\bvolume{24}
\bpages{729--740}.
\bptok{imsref}%
\end{barticle}
%
\endbibitem

\bibitem{Gbook}
%
\begin{bbook}[mr]
\bauthor{\bsnm{Girko},~\bfnm{V.~L.}\binits{V.~L.}}
(\byear{1990}).
\btitle{Theory of Random Determinants}.
\bseries{Mathematics and Its Applications (Soviet Series)}
\bvolume{45}.
\bpublisher{Kluwer Academic}, \blocation{Dordrecht}.
\bnote{Translated from the Russian}.
\bid{mr={1080966}}
\bptok{imsref}%
\end{bbook}
%
\endbibitem

\bibitem{G}
%
\begin{barticle}[mr]
\bauthor{\bsnm{Girko},~\bfnm{V.~L.}\binits{V.~L.}}
(\byear{1997}).
\btitle{A refinement of the central limit theorem for random determinants}.
\bjournal{Theory Probab. Appl.}
\bvolume{42}
\bpages{121--129}.
\bptnote{check year}%
\bptok{imsref}%
\end{barticle}
%
\endbibitem

\bibitem{Goodman}
%
\begin{barticle}[mr]
\bauthor{\bsnm{Goodman},~\bfnm{N.~R.}\binits{N.~R.}}
(\byear{1963}).
\btitle{The distribution of the determinant of a complex {W}ishart distributed
matrix}.
\bjournal{Ann. Math. Statist.}
\bvolume{34}
\bpages{178--180}.
\bid{issn={0003-4851}, mr={0145619}}
\bptok{imsref}%
\end{barticle}
%
\endbibitem

\bibitem{GZ}
%
\begin{barticle}[mr]
\bauthor{\bsnm{Guionnet},~\bfnm{A.}\binits{A.}} \AND
\bauthor{\bsnm{Zeitouni},~\bfnm{O.}\binits{O.}}
(\byear{2000}).
\btitle{Concentration of the spectral measure for large matrices}.
\bjournal{Electron. Commun. Probab.}
\bvolume{5}
\bpages{119--136 (electronic)}.
\bid{doi={10.1214/ECP.v5-1026}, issn={1083-589X}, mr={1781846}}
\bptok{imsref}%
\end{barticle}
%
\endbibitem

\bibitem{KKS}
%
\begin{barticle}[mr]
\bauthor{\bsnm{Kahn},~\bfnm{Jeff}\binits{J.}},
\bauthor{\bsnm{Koml{\'o}s},~\bfnm{J{\'a}nos}\binits{J.}} \AND
\bauthor{\bsnm{Szemer{\'e}di},~\bfnm{Endre}\binits{E.}}
(\byear{1995}).
\btitle{On the probability that a random {$\pm1$}-matrix is singular}.
\bjournal{J. Amer. Math. Soc.}
\bvolume{8}
\bpages{223--240}.
\bid{doi={10.2307/2152887}, issn={0894-0347}, mr={1260107}}
\bptok{imsref}%
\end{barticle}
%
\endbibitem

\bibitem{Kom}
%
\begin{barticle}[mr]
\bauthor{\bsnm{Koml{\'o}s},~\bfnm{J.}\binits{J.}}
(\byear{1967}).
\btitle{On the determinant of {$(0,1)$} matrices}.
\bjournal{Studia Sci. Math. Hungar.}
\bvolume{2}
\bpages{7--21}.
\bid{issn={0081-6906}, mr={0221962}}
\bptok{imsref}%
\end{barticle}
%
\endbibitem

\bibitem{Kom1}
%
\begin{barticle}[mr]
\bauthor{\bsnm{Koml{\'o}s},~\bfnm{J.}\binits{J.}}
(\byear{1968}).
\btitle{On the determinant of random matrices}.
\bjournal{Studia Sci. Math. Hungar.}
\bvolume{3}
\bpages{387--399}.
\bid{issn={0081-6906}, mr={0238371}}
\bptok{imsref}%
\end{barticle}
%
\endbibitem

\bibitem{NRR}
%
\begin{barticle}[mr]
\bauthor{\bsnm{Nyquist},~\bfnm{H.}\binits{H.}},
\bauthor{\bsnm{Rice},~\bfnm{S.~O.}\binits{S.~O.}} \AND
\bauthor{\bsnm{Riordan},~\bfnm{J.}\binits{J.}}
(\byear{1954}).
\btitle{The distribution of random determinants}.
\bjournal{Quart. Appl. Math.}
\bvolume{12}
\bpages{97--104}.
\bid{issn={0033-569X}, mr={0063591}}
\bptok{imsref}%
\end{barticle}
%
\endbibitem

\bibitem{Pre}
%
\begin{barticle}[mr]
\bauthor{\bsnm{Pr{\'e}kopa},~\bfnm{A.}\binits{A.}}
(\byear{1967}).
\btitle{On random determinants. {I}}.
\bjournal{Studia Sci. Math. Hungar.}
\bvolume{2}
\bpages{125--132}.
\bid{issn={0081-6906}, mr={0211439}}
\bptok{imsref}%
\end{barticle}
%
\endbibitem

\bibitem{RW}
%
\begin{barticle}[mr]
\bauthor{\bsnm{Rempa{\l}a},~\bfnm{Grzegorz}\binits{G.}} \AND
\bauthor{\bsnm{Weso{\l}owski},~\bfnm{Jacek}\binits{J.}}
(\byear{2005}).
\btitle{Asymptotics for products of independent sums with an
application to
{W}ishart determinants}.
\bjournal{Statist. Probab. Lett.}
\bvolume{74}
\bpages{129--138}.
\bid{doi={10.1016/j.spl.2005.04.034}, issn={0167-7152}, mr={2169371}}
\bptok{imsref}%
\end{barticle}
%
\endbibitem

\bibitem{Ro}
%
\begin{barticle}[mr]
\bauthor{\bsnm{Rouault},~\bfnm{Alain}\binits{A.}}
(\byear{2007}).
\btitle{Asymptotic behavior of random determinants in the {L}aguerre, {G}ram
and {J}acobi ensembles}.
\bjournal{ALEA Lat. Am. J. Probab. Math. Stat.}
\bvolume{3}
\bpages{181--230}.
\bid{issn={1980-0436}, mr={2365642}}
\bptok{imsref}%
\end{barticle}
%
\endbibitem

\bibitem{RV}
%
\begin{barticle}[mr]
\bauthor{\bsnm{Rudelson},~\bfnm{Mark}\binits{M.}} \AND
\bauthor{\bsnm{Vershynin},~\bfnm{Roman}\binits{R.}}
(\byear{2008}).
\btitle{The {L}ittlewood--{O}fford problem and invertibility of random
matrices}.
\bjournal{Adv. Math.}
\bvolume{218}
\bpages{600--633}.
\bid{doi={10.1016/j.aim.2008.01.010}, issn={0001-8708}, mr={2407948}}
\bptok{imsref}%
\end{barticle}
%
\endbibitem

\bibitem{SzT}
%
\begin{barticle}[auto:STB|2013/01/29|08:09:18]
\bauthor{\bsnm{Szekered},~\bfnm{G.}\binits{G.}} \AND
\bauthor{\bsnm{Tur{\'a}n},~\bfnm{P.}\binits{P.}}
(\byear{1937}).
\btitle{On an extremal problem in the theory of determinants (Hungarian)}.
\bjournal{Math. Naturwiss. Am. Ungar. Akad. Wiss.}
\bvolume{56}
\bpages{796--806}.
\bptok{imsref}%
\end{barticle}
%
\endbibitem

\bibitem{TVsurvey}
%
\begin{bmisc}[auto:STB|2013/01/29|08:09:18]
\bauthor{\bsnm{Tao},~\bfnm{T.}\binits{T.}} \AND
\bauthor{\bsnm{Vu},~\bfnm{V.}\binits{V.}}
\bhowpublished{Random matrices: The universality phenomenon for Wigner
ensembles. Available at arXiv:\arxivurl{1202.0068v1}.}
\bptok{imsref}%
\end{bmisc}
%
\endbibitem

\bibitem{TVdet}
%
\begin{barticle}[mr]
\bauthor{\bsnm{Tao},~\bfnm{Terence}\binits{T.}} \AND
\bauthor{\bsnm{Vu},~\bfnm{Van}\binits{V.}}
(\byear{2006}).
\btitle{On random {$\pm1$} matrices: Singularity and determinant}.
\bjournal{Random Structures Algorithms}
\bvolume{28}
\bpages{1--23}.
\bid{doi={10.1002/rsa.20109}, issn={1042-9832}, mr={2187480}}
\bptok{imsref}%
\end{barticle}
%
\endbibitem

\bibitem{TVsing}
%
\begin{barticle}[mr]
\bauthor{\bsnm{Tao},~\bfnm{Terence}\binits{T.}} \AND
\bauthor{\bsnm{Vu},~\bfnm{Van}\binits{V.}}
(\byear{2007}).
\btitle{On the singularity probability of random {B}ernoulli matrices}.
\bjournal{J.~Amer. Math. Soc.}
\bvolume{20}
\bpages{603--628}.
\bid{doi={10.1090/S0894-0347-07-00555-3}, issn={0894-0347}, mr={2291914}}
\bptok{imsref}%
\end{barticle}
%
\endbibitem

\bibitem{TVcir}
%
\begin{barticle}[mr]
\bauthor{\bsnm{Tao},~\bfnm{Terence}\binits{T.}} \AND
\bauthor{\bsnm{Vu},~\bfnm{Van}\binits{V.}}
(\byear{2008}).
\btitle{Random matrices: The circular law}.
\bjournal{Commun. Contemp. Math.}
\bvolume{10}
\bpages{261--307}.
\bid{doi={10.1142/S0219199708002788}, issn={0219-1997}, mr={2409368}}
\bptok{imsref}%
\end{barticle}
%
\endbibitem

\bibitem{TVhard}
%
\begin{barticle}[mr]
\bauthor{\bsnm{Tao},~\bfnm{Terence}\binits{T.}} \AND
\bauthor{\bsnm{Vu},~\bfnm{Van}\binits{V.}}
(\byear{2010}).
\btitle{Random matrices: The distribution of the smallest singular values}.
\bjournal{Geom. Funct. Anal.}
\bvolume{20}
\bpages{260--297}.
\bid{doi={10.1007/s00039-010-0057-8}, issn={1016-443X}, mr={2647142}}
\bptok{imsref}%
\end{barticle}
%
\endbibitem

\bibitem{TVsmooth}
%
\begin{barticle}[mr]
\bauthor{\bsnm{Tao},~\bfnm{Terence}\binits{T.}} \AND
\bauthor{\bsnm{Vu},~\bfnm{Van}\binits{V.}}
(\byear{2010}).
\btitle{Smooth analysis of the condition number and the least singular value}.
\bjournal{Math. Comp.}
\bvolume{79}
\bpages{2333--2352}.
\bid{doi={10.1090/S0025-5718-2010-02396-8}, issn={0025-5718}, mr={2684367}}
\bptok{imsref}%
\end{barticle}
%
\endbibitem

\bibitem{TVlocal}
%
\begin{barticle}[mr]
\bauthor{\bsnm{Tao},~\bfnm{Terence}\binits{T.}} \AND
\bauthor{\bsnm{Vu},~\bfnm{Van}\binits{V.}}
(\byear{2011}).
\btitle{Random matrices: Universality of local eigenvalue statistics}.
\bjournal{Acta Math.}
\bvolume{206}
\bpages{127--204}.
\bid{doi={10.1007/s11511-011-0061-3}, issn={0001-5962}, mr={2784665}}
\bptok{imsref}%
\end{barticle}
%
\endbibitem

\bibitem{Turan}
%
\begin{barticle}[mr]
\bauthor{\bsnm{Tur{\'a}n},~\bfnm{P.}\binits{P.}}
(\byear{1955}).
\btitle{On a problem in the theory of determinants}.
\bjournal{Acta Math. Sinica}
\bvolume{5}
\bpages{411--423}.
\bid{issn={0583-1431}, mr={0073555}}
\bptok{imsref}%
\end{barticle}
%
\endbibitem

\end{thebibliography}
\end{document}